\documentclass[11pt]{amsart}

\usepackage{color}
\usepackage[TS1,T1]{fontenc}
\usepackage[latin1]{inputenc}
\usepackage{amsfonts,amssymb,amsmath,amsthm,enumerate,latexsym,comment}

\newcommand{\fd}{\rightarrow}

\newcommand{\R}{\mathbb{R}}

\newcommand{\ito}{\infty}

\newcommand{\sm}{\setminus}

\newcommand{\la}{\lambda}

\newcommand{\Om}{\Omega}

\def\R{{\mathbb {R}}}

\newcommand{\ve}{\varepsilon}

\newtheorem{teo}{Theorem}[section]
\newtheorem{lema}[teo]{Lemma}

\theoremstyle{remark}
\newtheorem{remark}[teo]{Remark}

\theoremstyle{definition}
\newtheorem{defi}[teo]{Definition}

\numberwithin{equation}{section}

\begin{document}

\title[Three solutions for a nonlocal problem with critical growth]
{Three solutions for a nonlocal problem with critical
growth.}

\author[N. Cantizano and A. Silva]{Natal\'i Ail\'in Cantizano and Anal\'ia Silva}

\address{A. Silva
\hfill\break\indent Instituto de Matem\'atica Aplicada San Luis,
IMASL.
 \hfill\break\indent Universidad Nacional de San Luis and CONICET.
 \hfill\break\indent
Ejercito de los Andes 950.
 \hfill\break\indent D5700HHW San Luis, Argentina.}
 \email{{\tt asilva@dm.uba.ar }}
\urladdr[A.Silva]{https://analiasilva.weebly.com/}

\address{N. Cantizano
\hfill\break\indent Instituto de Matem\'atica Aplicada San Luis,
IMASL.
 \hfill\break\indent Universidad Nacional de San Luis and CONICET.
 \hfill\break\indent
Ejercito de los Andes 950.
 \hfill\break\indent D5700HHW San Luis, Argentina.}
 \email{{\tt ncantizano@unsl.edu.ar}}



\begin{abstract}
 The main  goal of this work  is to prove the existence of three different  solutions (one positive, one negative and one with nonconstant sign) for the equation $(-\Delta_p)^s u= |u|^{p^{*}_s -2} u +\lambda f(x,u)$  in a bounded domain with Dirichlet condition, where $(-\Delta_p)^s$ is the well known $p$-fractional Laplacian and
$p^*_s=\frac{np}{n-sp}$  is the critical Sobolev exponent for the
non local case. The proof follows the ideas of \cite{Silva} and is
based in the extension of the Concentration Compactness Principle
for the $p$-fractional Laplacian \cite{MS} and Ekeland's variational
Principle \cite{Ekeland}.
\end{abstract}

\maketitle

\section{Introduction}

Let us consider the following  non local equation with Dirchlet boundary conditions
\begin{equation}\label{P}
\begin{cases}
(-\Delta_{p})^s u= |u|^{p^{*}_s -2} u +\lambda f(x,u) & \text{in } \Omega,\\
u=0 & \text{in } \R^n \sm \Omega.
\end{cases}\,\,
\end{equation}
where $s\in(0,1)$, $\Omega$ is a smooth and bounded domain in
$\R^n$ and $(-\Delta_p)^su$, called the $p$-fractional Laplacian, is
defined up to a normalization constant by
\begin{equation*}
(-\Delta_p)^s u:=2 \, \lim_{\ve\fd 0^+} \int_{\R^{n}\sm B_{\ve}(x)}
\frac{|u(x)-u(y)|^{p-2}(u(x)-u(y))}{|x-y|^{n+ps}} \,dy  .
\end{equation*}
When $p=2$ this is the well known fractional Laplacian.
 Problems involving non local operators have many applications, just to cite a few, we refer to  \cite{DGLZ1, Eri, G-L} for some physical models, \cite{A-B, Leve, Sch} for some applications in finances, \cite{Co} for applications in fluid dynamics, \cite{Hu, Ma-Va, Re-Rho} for application in ecology and \cite{G-O} for some applications in image processing.

The functional framework for this operator are the fractional order
Sobolev spaces, see \cite{T} and \cite{H}. The fractional order
Sobolev space is defined by
$$
W^{s,p}(\R^n) := \left \{u\in L^p(\R^n)\colon [u]_{s,p}<\ito\right
\},
$$
where $[u]_{s,p}$ is the famous seminorm of Gagliardo is defined by
\begin{equation*}
    [u]_{s,p} := \left( \int_{\R^{2n}}\frac{(u(x)-u(y))^p}{|x-y|^{n+ps}}\, dx\,\,dy
    \right)^{\frac{1}{p}},
\end{equation*}
and $W^{s,p}_0(\Om)$ is defined by $ W^{s,p}_0(\Om):= \{u\in
L^{p}(\R^n) : [u]_{s,p}<\ito , u=0 \text{ in } \R^n\sm \Om \}. $
 It is well-known that when $sp<n$  the following Sobolev
inequality holds
$$
\left(\int_{\R^n} |u|^{\frac{np}{n-sp}}\, dx\right)^\frac{n-sp}{n}
\le C \int_{\R^{2n}} \frac{|u(x)-u(y)|^p}{|x-y|^{n+sp}}\,
dxdy
$$
for  $u\in C^\infty_c(\R^n)$, where $p^{*}_s=\frac{np}{n-sp}$ is
called the critical Sobolev exponent. So, the embedding $W^{s,p}(\Omega)\hookrightarrow L^{q}(\Omega)$ for $1\leq q \leq p^{*}_s $ is continuous. Moreover, is compact for $ 1\leq q < p^{*}_s  $.
Critical equations with the fractional Laplacian in bounded domains
have been considered in \cite{BCPS,S2,S1,SV1,SV2}. Multiplicity of
solutions for nonlocal equation with critical growth was studied in
\cite{FMS,PSY}. The main goal of this paper is to show the existence
of three different solutions of the problem \eqref{P}. Moreover
these solutions are one positive, one negative and one with non
constant sign. We impose adequate conditions on the source $f$ and
on the parameter $\lambda$ but we do not impose any parity
conditions on the source $f$. This result extends an old paper of
Struwe \cite{st}. Similar results for some local operators can be
found in \cite{DPFBS,Silva,FB,Li}.
 The method  in the proof used in \cite{st} consists on restricting the functional associated to \eqref{P} to three different manifolds constructed by imposing a sign restriction and normalizing condition. Then using Ekeland variational principle (see \cite{Ekeland}) and a generalization to the fractional setting obtained by Mosconi et al. for any $1<p<\tfrac{n}{s}$ (see
 \cite{MS})
of the well known Concentration Compactness Principle of P.L.Lions
(see \cite{Lions}), we can prove the existence of a critical point
of each restricted functional, that are critical points of the
unrestricted one.

Throughout this work, by weak solution of
\eqref{P} we understand critical points of the associated energy functional acting on the Sobolev space $ W^{s,p}_0 (\Omega ) $:
\begin{equation}\label{energia}
\Phi(u)= \frac{1}{p}\int_{\R^{2n}}
\frac{\left(u(x)-u(y)\right)^p}{|x-y|^{n+ps}} \, dy \,dx -
\int_{\Omega} \frac{1}{p^*_s} |u(x)|^{p^{*}_s} +\lambda F(x,u(x))
\,dx,
\end{equation}
where $F(x,u)=\int_0^u f(x,z) dz$.

\section{Assumptions and statement of the results}

The precise assumptions on the source terms $f$ are as follows:
\begin{itemize}
    \item [(H1)] $f :\Om \times \R \fd \R$, is a measurable function with respect to the first argument and continuously
differentiable with respect to the second argument for almost every
$x  \in  \Om$. Moreover, $f (x, 0) = 0$ for every $x \in  \Om$.
    \item  [(H2)] There exist constants $c_1\in (0,\frac{1}{p^*_s -1})$,  $c_2\in (p,p^*_s)$, $0<c_3<c_4$ such that for any $u\in L^q (\Omega )$ and $p<q<p^*_s$,
    \begin{equation*}
        c_3\|u\|^q_{L^q(\Omega)} \leq c_2 \int_{\Omega} F(x,u)\,dx \leq \int_{\Omega} f(x,u)u\,dx \leq c_1 \int_{\Omega} f_u(x,u)u^2\,dx \leq  c_4\|u\|^q_{L^q(\Omega)}.
    \end{equation*}
\end{itemize}

\begin{remark} The following example fulfill all of our
hypotheses, $f(x,u)=|u|^{q-2} u + |u_{+}|^{r-2} u_{+}$ if $r\leq q $.
\end{remark}
So the main result of the paper reads:
\begin{teo}\label{main}
Under the assumptions $(H1)-(H2)$, there exist $\la^*>0$ depending
only on $n,\, p,\, q$ and the constant $c_3$ in $(H2)$, such that
for every $\la>\la^*$, there exist three different, nontrivial,
(weak) solutions of problem \eqref{P}. Moreover these solutions are,
one positive, one negative and the other one has non-constant sing.
\end{teo}

\section{Proof of Theorem \ref{main} }

We will construct three disjoint sets $K_i$ not containing $0$ such
that $\Phi $ has a critical point in $K_i$. These sets will be
subsets of $C^{1}-$manifolds $M_i \subset W^{s,p}_0(\Omega)$ that will
be constructed by imposing a sing restriction and a normalizing
condition.

In fact,
\begin{defi} For each $i=1,2,3$, let $M_i\subset W^{s,p}_0 (\Omega) $ be defined as
\begin{equation*}
M_1=\left\{u\in W^{s,p}_0(\Omega) : \int_{\Omega} u_+>0 \,\,and \,\,
[u_+]_{s,p}^p- \int_{\Omega}  |u_+|^{p^{*}_s} \,dx=\int_{\Omega}\lambda
f(x,u) u_+\,dx \right\},
\end{equation*}
\begin{equation*}
M_2=\left\{u\in W^{s,p}_0(\Omega) : \int_{\Omega} u_->0 \,\,and \,\,
[u_-]_{s,p}^p- \int_{\Omega}  |u_-|^{p^{*}_s}\,dx =\int_{\Omega}\lambda
f(x,u)u_-\, dx \right\},
\end{equation*}
\begin{equation*}
M_3=M_1 \cap M_2,
\end{equation*}
where $u_+= \max\{u,0\}$ and $u_-=\max \{-u,0\}$.
\end{defi}

\begin{defi}For each $i=1,2,3$, let $K_i\subset W^{s,p}_0 (\Omega) $ be defined as
\begin{equation*}
K_1=\{ u\in M_1: u\geq 0 \}, \,\,\, K_2=\{ u\in M_2: u\leq 0 \},
\,\,\, K_3=M_3.
\end{equation*}
\end{defi}

First, we need the following lemma to show that these sets are
nonempty and, moreover, give some properties that will be useful in
the proof of our main result.

\begin{lema}\label{t}
For every $w_0\in W^{s,p}_0(\Omega)$, $w_0>0$ ($w_0<0$), there exists
$t_\la$>0 such that $t_\la w_0\in M_1$ ($\in M_2$). Moreover,
$\lim_{\la \fd \ito} t_{\la}=0$.

As a consequence, given $w_0,w_1 \in W^{s,p}_0(\Omega)$, $w_0>0$,
$w_1<0$ with disjoint supports, there exist
$\overline{t}_{\la},\underline{t}_{\la}$ such that
$\overline{t}_{\la}w_0+\underline{t}_{\la}w_1 \in M_3$. Moreover
$\overline{t}_{\la},\underline{t}_{\la} \fd 0$ as $\la \fd \ito$.
\end{lema}

\begin{proof}
We prove Lemma \ref{t} for $M_1$, the other cases are analogous.

For $w\in W^{s,p}_0(\Omega)$, $w\geq 0$, we consider the functional
\begin{equation*}
\varphi_1(w)= [w]_{s,p}^p- \int_{\Omega} |w|^{p^{*}_s} +\lambda f(x,w)w
\,dx.
\end{equation*}

Given $w_0$, in order to prove the lemma, we must show that
$\varphi_1(t_\la w_0)=0$ for some $t_\la$. Using the hypothesis
(H2), we have that:
\begin{equation*}
    \varphi_1 (t w_0) \geq A t^p- Bt^{p^{*}_s} - \la c_4 Et^q
\end{equation*}
and
\begin{equation*}
    \varphi_1 (t w_0) \leq A t^p- Bt^{p^{*}_s} - \la c_3 Et^q,
\end{equation*}
where  the coefficients A,B and E are given by:
\begin{equation*}
A= [w_0]_{s,p}^p\, ,\,\,B= \int_{\Omega} |w_0|^{p^{*}_s}\,dx ,\,\,
E=\int_{\Omega} |w_0|^{q}\,dx.
\end{equation*}

Since $p<q<p^*_s$ it follows that $\varphi_1 (t w_0)$ is positive
for a $t$ small enough, and negative for $t$ big enough. Hence, by
Bolzano\rq s Theorem, there exists some $t=t_\la$ such that
$\varphi_1 (t_\la w_0)=0$.

In order to give an upper bound for $t_\la$, it is enough to find
some $t_1$, such that $\varphi_1 (t_1 w_0)<0$. We observe that:
\begin{equation*}
 \varphi_1 (t w_0) < A t^p - \la c_3 Et^q,
\end{equation*}
so it is enough to choose $t_1$ such that $At^p_1 - \la c_3 E
t^q_1=0$, i.e.,
\begin{equation*}
 t_1=\left( \frac{A}{c_3\la E} \right)^{\frac{1}{(q-p)}},
\end{equation*}
therefore, again by Bolzano\rq s Theorem, we can choose $t_\la \in
[0,t_1]$, which implies that $t_\la \fd 0$ when $\la \fd +\ito $, as
we wanted to prove.
\end{proof}

For the proof of Theorem \ref{main}, we need also the following
lemmas.

\begin{lema}\label{cotasunif}
 There exist constants $\alpha_j>0$ such that, for every $u\in K_i$, $i=1,2,3,$
\begin{equation*}
 \alpha_1 [u]_{s,p}^p \leq \alpha_2\left(  \int_{\Omega} |u|^{p^{*}_s} +\lambda f(x,u)u \,dx \right) \leq \alpha_3 \Phi (u) \leq \alpha_4
 [u]_{s,p}^p.
\end{equation*}
\end{lema}

\begin{proof}
 As $u\in K_i$, we have that
 \begin{equation*}
      [u]_{s,p}^p =\int_{\Omega} |u|^{p^{*}_s} +\lambda f(x,u)u \,dx,
 \end{equation*}
choosing $\alpha_1=\alpha_2$ we have the first inequality.

For the last inequality by (H2)
\begin{equation*}
    \int_{\Omega} F(x,u)\,dx \leq \frac{1}{c_2} \int_{\Omega}  f(x,u)u\,dx.
\end{equation*}

 Furthermore,
 \begin{equation*}
   \left| \la \int_{\Omega} F(x,u)\,dx \right| = \la \int_{\Omega} F(x,u)\,dx \leq  \frac{1}{c_2}\int_{\Omega} \la f(x,u)u \,dx= \frac{1}{c_2}\left([u]_{s,p}^p -\int_{\Omega} |u|^{p^{*}_s}\,dx \right),
 \end{equation*}
 so
  \begin{equation}\label{cristiano}
   -\la \int_{\Omega} F(x,u)\,dx \leq  \frac{1}{c_2}\left( [u]_{s,p}^p -\int_{\Omega} |u|^{p^{*}_s}\,dx
   \right).
 \end{equation}

 By \ref{cristiano}, we have:
\begin{align*}
     \Phi(u)=& \frac{1}{p} [u]_{s,p}^p - \int_{\Omega} \frac{1}{p^*_s} |u(x)|^{p^{*}_s} +\lambda F(x,u) \,dx\\
     \leq& \frac{1}{p} [u]_{s,p}^p- \int_{\Omega} \frac{1}{p^*_s} |u(x)|^{p^{*}_s} \,dx +\frac{1}{c_2}\left( [u]_{s,p}^p -\int_{\Omega} |u(x)|^{p^{*}_s}\,dx \right)\\
     \leq&  \frac{1}{p} [u]_{s,p}^p +\frac{1}{c_2} [u]_{s,p}^p \\
     \leq& \left( \frac{1}{p} +\frac{1}{c_2} \right)[u]_{s,p}^p.
\end{align*}
This proves the third inequality, with $\alpha_4=\left(\frac{1}{p}
+\frac{1}{c_2} \right)\alpha_3$.

To prove the middle inequality we proceed as follows:
 \begin{align*}
     \Phi(u)=& \frac{1}{p} [u]_{s,p}^p- \int_{\Omega} \frac{1}{p^{*}_s} |u(x)|^{p^{*}_s} +\lambda F(x,u) \,dx\\
     \geq& \frac{1}{p} [u]_{s,p}^p-\int_{\Omega}   \frac{1}{p^{*}_s} |u(x)|^{p^{*}_s}\,dx - \frac{1}{c_2} \int_{\Omega} \la f(x,u)u \,dx.
 \end{align*}
 So
 \begin{align*}
     c_2 \Phi(u) \geq& c_2 \frac{1}{p}[u]_{s,p}^p -c_2 \int_{\Omega}   \frac{1}{p^*_s} |u(x)|^{p^{*}_s}\,dx -  \int_{\Omega}\la  f(x,u)u \,dx\\
     =& c_2 \frac{1}{p}\left(\int_{\Omega}   |u(x)|^{p^{*}_s}\,dx +  \int_{\Omega}  \la f(x,u)u\, dx \right) - c_2 \int_{\Omega}   \frac{1}{p^{*}_s} |u(x)|^{p^{*}_s}\,dx -  \int_{\Omega}\la  f(x,u)u\, dx\\
     =& c_2 \frac{1}{p}\int_{\Omega}   |u(x)|^{p^{*}_s}dx + c_2 \frac{1}{p} \int_{\Omega}  \la f(x,u)u\, dx  - c_2 \int_{\Omega}   \frac{1}{p^{*}_s} |u(x)|^{p^{*}_s}\,dx -  \int_{\Omega}\la  f(x,u)u\, dx\\
     =& c_2 \left(  \frac{1}{p}- \frac{1}{p^{*}_s} \right) \int_{\Omega}   |u(x)|^{p^{*}_s}\, dx + \left( c_2 \frac{1}{p} - 1 \right) \int_{\Omega}  \la f(x,u)u\,dx.
 \end{align*}
 Since $\gamma_1= c_2 \left(  \frac{1}{p}- \frac{1}{p^{*}_s}\right)$ and $\gamma_2= \left( c_2 \frac{1}{p}-1\right)$ are positive, we take $\alpha_2= \min \{\gamma_1, \gamma_2 \}$, $\alpha_3=c_2$ and we have
 \begin{equation*}
     \alpha_3 \Phi (u) \geq \alpha_2 \left(  \int_{\Omega} |u|^{p^{*}_s} +\lambda f(x,u)u\,dx \right).
 \end{equation*}
This finishes the proof.
\end{proof}

\begin{lema}\label{cota} There exists a constant
$ D $ such that $ [u_+]^p_s\geq D$, for all  $ u \in K_1$,
$[u_-]^p_{s,p}\geq D$ for all $  u \in K_2$, and $[u_-]^p_{s,p}\, ,
\,[u_+]^p_s\geq D$ for all $ u \in K_3$.
\end{lema}

\begin{proof}

By definition of $K_i$ we have
\begin{equation*}
 [u_\pm]^p_{s,p}= \|u_\pm \|_{p^{*}_s}^{p^{*}_s} + \int_\Omega \la
f(x,u)u_\pm \,dx .
\end{equation*}
Using (H2) we have
\begin{equation*}
\int_\Omega \la f(x,u)u_\pm \,dx \leq c_4\|u_\pm \|_q^q \text{, for
} p^{*}_s\geq q >p.
\end{equation*}
Then
\begin{equation*}
 [u_\pm]^p_{s,p}\leq \|u_\pm\|_{p^{*}_s}^{p^{*}_s} + c_4\|u_\pm\|_q^q
\leq \tilde{C} \left(  [u_\pm]^{p^{*}_s}_{s} + [u_\pm]_s^q \right) .
\end{equation*}
In the second inequality we use Poincar\'e inequality.
In summary $[u_+]^p_{s,p}\leq \hat{C}[u_\pm]_{s,p}^r.$  Where $r=q$ if
$[u_\pm]_{s,p}<1$ or $r=p^*_s$ if $[u_\pm]_{s,p}\geq 1$. Since $r>p$ we
have what we need.
\end{proof}

The following lemma describes the properties of the manifolds $M_i$.

\begin{lema}\label{variedad}
$M_i$ is a sub-manifold of $W^{s,p}_0(\Omega)$ of codimension 1, if
$i=1,2$ and 2 if $i=3$ respectively, the sets $K_i$ are complete, and for every $ u \in M_i$ we have $T_u W^{s,p}_0(\Omega)=T_u M_i\oplus\mbox{span}\{u_+,u_-\}$ where $T_uM$ is the tangent space at u of the Banach manifold M. Finally, the projection to the first
coordinate is uniformly continuous on $M_i$.
\end{lema}

\begin{proof}We consider
\begin{equation*}
\overline{M_1}=\left\{u\in W^{s,p}_0(\Omega) : \int_{\Omega} u_+>0
\right\},
\end{equation*}
\begin{equation*}
\overline{M_2}=\left\{u\in W^{s,p}_0(\Omega) : \int_{\Omega}
u_->0\right\},
\end{equation*}
\begin{equation*}
\overline{M_3}=\overline{M_1} \cap \overline{M_2}.
\end{equation*}

Observe that $M_i\subset \overline{M_i}$ and since the sets
$\overline{M_i}$ are open so it's sufficient to prove that $M_i$ is
a regular sub-manifold of $W^{s,p}_0(\Omega)$.

We are going to build a function $C^{1}$,
$\varphi:\overline{M_i}\fd\R^d$ with $d=1$ if $i=1,2$ or $d=2$ if
$i=3$, such that $M_i$ is the inverse of a regular value of $ \varphi_i $.

We define
\begin{equation*}
\varphi_1(u)=  [u_+]_{s,p}^p - \int_{\Omega} |u_+|^{p^{*}_s} + \lambda
f(x,u)u_+ \,dx\,\,\, \text{ for } u\in M_1,
\end{equation*}
\begin{equation*}
\varphi_2(u)=  [u_-]_{s,p}^p - \int_{\Omega} |u_-|^{p^{*}_s} +\lambda
f(x,u)u_- \,dx\,\,\, \text{ for } u\in M_2,
\end{equation*}
and
\begin{equation*}
\varphi_3(u)=(\varphi_1(u),\varphi_2(u))  \,\,\, \text{ for } u\in
M_3.
\end{equation*}

We have that  $M_i=\varphi^{-1}_i(0)$ so we have to prove that 0 is
a regular value of $\varphi_i$.

Let us calculate $\langle\nabla \varphi_1(u),u_+\rangle$ for $u\in
M_1$,
\begin{equation*}
\frac{d}{d\ve}\varphi_1(u+\ve u_+)= \frac{d}{d\ve}\left([(u+\ve
u_+)_+]_{s,p}^p - \int_{\Omega}  |(u+\ve u_+)_+|^{p^{*}_s} +\lambda
f(x,u +\ve u_+)(u+\ve u_+)_+ \,dx\right).
\end{equation*}
Since $(u+\ve u_+)_+=u_++\ve u_+$ we have that
$\frac{d}{d\ve}\varphi_1(u+\ve u_+)$ is equal to
\begin{equation*}
  (1+\ve)^{p-1} p[u_+]_{s,p}^p - \int_{\Omega}  p^{*}_s  (1+\ve)^{p_s^{*}-1} |u_+|^{p^{*}_s} +\lambda f(x,u +\ve u_+)u_+ +\lambda f_u(x,u +\ve u_+)(1+\ve)u_+^2\, dx,
\end{equation*}
then since $u\in M_1$,
\begin{align*}
    \frac{d}{d\ve}\varphi_1(u+\ve u_+)  \Big|_{\ve=0}=& \left(
 p\, [u_+]_{s,p}^p - \int_{\Omega} p^{*}_s |u_+|^{p^{*}_s} +\lambda
f(x,u)u_+ +\lambda f_u(x,u)u_+^2 \,dx\right)\\
\leq & \,p^{*}_s\left([u_+]_{s,p}^p- \int_{\Omega} |u_+|^{p^{*}_s}\, dx\right)
-\int_{\Omega} \lambda f(x,u)u_+ +\lambda f_u(x,u)u_+^2 \,dx\\
=&\, p^{*}_s\left(\int_{\Omega} \lambda f(x,u)u_+ dx \right)
-\int_{\Omega}  \lambda f(x,u)u_+ +\lambda f_u(x,u)u_+^2 \,dx\\
=& \, ( p^{*}_s -1 )\left(\int_{\Omega} \lambda f(x,u)u_+\, dx \right)
-\int_{\Omega}\lambda f_u(x,u)u_+^2 \,dx.
\end{align*}

By $(H_2)$ we know that there exists $c_1 \in \left( 0,
\frac{1}{p^{*}_s -1} \right) $ such that
\begin{equation}\label{shakira}
\int_{\Omega} \lambda f(x,u)u_+ \,dx \leq c_1 \int_{\Omega}\lambda
f_u(x,u)u_+^2 \,dx.
\end{equation}

Then
\begin{equation*}
    (p^{*}_s -1) \int_{\Om} \la f(x,u) u_+ \,dx - \int_{\Om} \la f_u (x,u)u^2_+ \,dx <0.
\end{equation*}

In summary, we have that $\langle\nabla \varphi_1(u),u_+\rangle <0$,
then $\nabla\varphi_1(u)\neq 0$. This means that $M_1$ is a regular
submanifold of $W^{s,p}_0(\Omega )$.

The proof for $M_2$, is analogous.

Let's observe that if we prove that $\langle\nabla
\varphi_2(u),u_+\rangle =\langle\nabla \varphi_1(u),u_-\rangle=0$
for $u\in M_3$ then for what we had made before, we know that
$\langle\nabla \varphi_1(u),u\rangle<0$ and $\langle\nabla
\varphi_2(u),u\rangle<0$. For this we can affirm that
$\nabla\varphi_3(u)\neq 0$ for $u\in M_3$.

Then we will prove that $\langle\nabla\varphi_1(u),u_-\rangle=0$. In
fact,
\begin{align*}
    \frac{d}{d\ve}\varphi_1(u+\ve u_-)=& \frac{d}{d\ve}\left([(u+\ve u_-)_+]_{s,p}^p - \int_{\Omega} |(u+\ve u_-)_+|^{p^{*}_s} +\lambda f(x,u +\ve u_-)(u+\ve u_-)_+\, dx\right)\\
    =& \frac{d}{d\ve}\left( [u_+]_{s,p}^p -\int_{\Omega}   |u_+|^{p^{*}_s} +\lambda f(x,u +\ve u_-)u_+ \,dx\right)\\
    =& -\int_{\Omega} \la f_u (x,u +\ve u_-)u_+ u_-\, dx =0.
\end{align*}
Then
\begin{equation*}
     \left. \frac{d}{d\ve}\varphi_1(u+\ve u_-)\right|_{\ve=0} =0.
\end{equation*}
In an analogous way we have $\langle\nabla \varphi_2(u),u_+\rangle
=0$. Therefore, $M_3$ is a regular submanifold.

The completeness of $K_i$ is easy and is left to the reader.

Finally, it remains to see that
\begin{equation*}
    T_u W^{s,p}_0(\Omega)=T_u M_1 \oplus\mbox{span}\{u_+\},
\end{equation*}
where $M_1=\{ u: \varphi_1(u)=0 \}$ and $T_u M_1 =\{ v:
\langle\nabla \varphi_1(u),v\rangle =0 \}$. Now let $v\in T_u
W^{s,p}_0(\Omega)$ be a unit tangential vector, then $v=v_1+v_2$
where $v_2=\alpha u_+$ and $v_1=v-v_2$. Let us take $\alpha$ as
\begin{equation*}
    \alpha=\frac{\langle\nabla \varphi_1(u),v\rangle}{\langle\nabla \varphi_1(u),u_+\rangle}.
\end{equation*}
With this choice, we have that $v_1\in T_u M_1$. Now
\begin{equation*}
     \langle\nabla \varphi_1(u),v_1\rangle =0 .
\end{equation*}
The very same argument is used to show that $T_u
W^{s,p}_0(\Omega)=T_u M_2 \oplus\mbox{span}\{u_-\}$ and $T_u
W^{s,p}_0(\Omega)=T_u M_i \oplus\mbox{span}\{u_+,u_-\}$.

From these formulas and the estimates given in the first part of the
proof, the uniform continuity of the projections onto $T_u M_i$
follows.
\end{proof}

Now, we say that $\{u_j\}\subset W^{s,p}_0(\Om)$ is a Palais-Smale
sequence of $c$ level if
\begin{itemize}
 \item[(i)] $\Phi(u_j)\to c$,
 \item[(ii)] $\nabla\Phi(u_j)\to 0$ in $W^{-s,p}(\Om)$.
\end{itemize}
We say that $\Phi$ satisfies Palais-Smale condition of level $c$ if
for every  $\{u_j\}$ Palais-Smale sequence of level $c$ there exists
a subsequence that converges strongly in $W^{s,p}_0(\Om)$.

Now, in order to use Ekeland's variational principle, we need to
check the Palais-Smale condition for the functional $\Phi$
restricted to the manifold $M_i$. To this end, we need the following
lemma which proves the Palais-Smale condition for the unrestricted
functional below certain energy level.

\begin{lema}\label{constante_S}
The unrestricted functional $\Phi$ verifies the Palais-Smale
condition for energy level c for every
$c<\frac{s}{n}S^{\frac{n}{sp}}$, where $S$ is the best Sobolev
constant for the fractional Laplacian $S:=\inf_{\phi \in C^{\ito}_c
(\Omega)} \frac{ [\phi]^p_{s,p} }{\|\phi\|^p_{p^{*}_s}}.$

\end{lema}
The proof of Lemma \ref{constante_S} is omitted as it uses standard
ideas and is based in the Concentration Compactness Principle for
nonlocal operators (see\cite{MS}). For the local case it can be
found in \cite{GAP,Silva}. For the non local case it follows
similarly, see \cite{FBSSRN} for the details.

Now, we can prove the Palais-Smale condition for the restricted
functional.

\begin{lema}\label{condicion}
The functional $\Phi|_{K_i}$ satisfies the Palais-Smale condition
for energy level $c$ for every $c<\frac{s}{n} S^{\frac{n}{sp}}$.
\end{lema}

\begin{proof}
Let $\{ u_k \}\subset K_i$ be a Palais-Smale sequence, that is
$\Phi(u_k)$ is uniformly bounded and $\nabla \Phi |_{K_i} \fd 0$
strongly. We need to show that there exists a subsequence $u_{k_j}$
that converges strongly in $K_i$.

Let $v_j\in T_{u_j} W^{s,p}_0(\Omega)$ be a unit tangential vector
such that
\begin{equation*}
    \langle\nabla \Phi (u_j) ,v_j\rangle =\|\nabla\Phi
    (u_j)\|_{W^{-s,p}_0 (\Omega)}.
\end{equation*}
Now, by lemma \ref{variedad}, $v_j=w_j+z_j$ with $w_j\in T_{u_j}M_i
$ and $z_j\in span\{(u_j)_+,(u_j)_-\}$.

Since $\Phi(u_j)$ is uniformly bounded, by Lemma \ref{cotasunif},
$u_j$ is uniformly bounded in $W^{s,p}_0(\Omega)$ and hence $w_j$ is
uniformly bounded in $W^{s,p}_0(\Omega)$. Therefore
\begin{equation*}
    \|\nabla\Phi (u_j)\|_{W^{-s,p}_0(\Omega)}=\langle\nabla \Phi (u_j) ,v_j\rangle =\langle\nabla \Phi|_{K_i} (u_j) ,v_j\rangle \fd
    0.
\end{equation*}

As $v_j$ is uniformly bounded and $\nabla \Phi|_{K_i} (u_j)\fd 0$
strongly, the inequality converges strongly to $0$. Now the result
follows by Lemma \ref{constante_S}.

\end{proof}

We now immediately obtain the following lemma.
\begin{lema}
There exists $u\in K_i$ be a critical point of the restricted
functional $\Phi|_{K_i}$. Moreover $u$ is also a critical point of
the unrestricted functional $\Phi$ and hence a weak solution to
\eqref{P}.
\end{lema}

With all this preparatives, this is the proof of our main result.

\begin{proof}[Proof of Theorem \ref{main}]
To prove the Theorem \ref{main}, we need to check that the functional $\Phi|_{K_i}$ verifies the hypotheses of the Ekeland's Variational Principle.

The fact that $\Phi$ is bounded below over $K_i$ is a direct
consequence of the construction of the manifold $K_i$.

Then by Ekeland's Variational Principle, there exists $v_k\in K_i$,
such that
$$\Phi (v_k)\fd c_i \text{    and   } (\nabla\Phi|_{K_i})(v_k)\fd 0. $$

We have to check that if we choose $\lambda$ large, we have that
$c_i<\frac{s}{n} S^{\frac{n}{sp}}$. This follows easily from Lemma
\ref{t}. For instance, for $c_1$ we have that choosing $w_0\geq 0$,
$$
c_1 \leq \Phi(t_\lambda w_0) \leq \frac{1}{p}
t_{\lambda}^p[w_0]_{s,p}^p.
$$

 Moreover, it follows from the
estimate of $t_\lambda$ in Lemma \ref{t} , that $c_1 \to 0$ as
$\lambda \to 0$. Then $c_i<\frac{s}{n} S^{\frac{n}{sp}}$ for
$\lambda>\lambda^*(p,q,n,c_3)$. The other cases are analogous.

From Lemma \ref{constante_S}, it follows that $v_k$ has a convergent
subsequence, that we still call $v_k$. Therefore $\Phi$ has a
critical point in $K_i$, $i=1,2,3$ and, by construction, one of them
is positive, other is negative and the last one changes sign.
\end{proof}

\section*{Acknowledgements}
We want to thank Juli\'an Fern\'andez Bonder  for his valuable help.

This paper was supported by grants UBACyT 20020130100283BA, CONICET
PIP 11220150100032CO and PROICO 031906, UNSL.

 N.Cantizano is a fellow of CONICET and A.Silva is member of CONICET.

\bibliographystyle{plain}
\bibliography{biblio}

\def\ocirc#1{\ifmmode\setbox0=\hbox{$#1$}\dimen0=\ht0 \advance\dimen0
  by1pt\rlap{\hbox to\wd0{\hss\raise\dimen0
  \hbox{\hskip.2em$\scriptscriptstyle\circ$}\hss}}#1\else {\accent"17 #1}\fi}
  \def\ocirc#1{\ifmmode\setbox0=\hbox{$#1$}\dimen0=\ht0 \advance\dimen0
  by1pt\rlap{\hbox to\wd0{\hss\raise\dimen0
  \hbox{\hskip.2em$\scriptscriptstyle\circ$}\hss}}#1\else {\accent"17 #1}\fi}
\begin{thebibliography}{10}

\bibitem{A-B}
Vedat Akgiray and G.~Geoffrey Booth.
\newblock The siable-law model of stock returns.
\newblock {\em Journal of Business \& Economic Statistics}, 6(1):51--57, 1988.

\bibitem{BCPS}
B.~Barrios, E.~Colorado, A.~de~Pablo, and U.~S\'anchez.
\newblock On some critical problems for the fractional {L}aplacian operator.
\newblock {\em J. Differential Equations}, 252(11):6133--6162, 2012.

\bibitem{Co}
Peter Constantin.
\newblock Euler equations, {N}avier-{S}tokes equations and turbulence.
\newblock In {\em Mathematical foundation of turbulent viscous flows}, volume
  1871 of {\em Lecture Notes in Math.}, pages 1--43. Springer, Berlin, 2006.

\bibitem{DPFBS}
Pablo~L. De~N{\'a}poli, Juli{\'a}n~Fern{\'a}ndez Bonder, and Anal{\'{\i}}a
  Silva.
\newblock Multiple solutions for the {$p$}-{L}aplace operator with critical
  growth.
\newblock {\em Nonlinear Anal.}, 71(12):6283--6289, 2009.

\bibitem{H}
Eleonora Di~Nezza, Giampiero Palatucci, and Enrico Valdinoci.
\newblock Hitchhiker's guide to the fractional {S}obolev spaces.
\newblock {\em Bull. Sci. Math.}, 136(5):521--573, 2012.

\bibitem{DGLZ1}
Qiang Du, Max Gunzburger, R.~B. Lehoucq, and Kun Zhou.
\newblock Analysis and approximation of nonlocal diffusion problems with volume
  constraints.
\newblock {\em SIAM Rev.}, 54(4):667--696, 2012.

\bibitem{Ekeland}
I.~Ekeland.
\newblock On the variational principle.
\newblock {\em J. Math. Anal. Appl.}, 47:324--353, 1974.

\bibitem{Eri}
A.~Cemal Eringen.
\newblock {\em Nonlocal continuum field theories}.
\newblock Springer-Verlag, New York, 2002.

\bibitem{FB}
Juli{\'a}n Fern{\'a}ndez~Bonder.
\newblock Multiple positive solutions for quasilinear elliptic problems with
  sign-changing nonlinearities.
\newblock {\em Abstr. Appl. Anal.}, 2004(12):1047--1055, 2004.

\bibitem{FBSSRN}
Juli{\'a}n Fern{\'a}ndez~Bonder, Nicolas Saintier, and Anal{\'{\i}}a Silva.
\newblock The concentration-compactness principle for fractional order sobolev
  spaces in unbounded domains and applications to the generalized fractional
  brezis-nirenberg problem.
\newblock https://arxiv.org/abs/1802.09322.

\bibitem{FMS}
Alessio Fiscella, Giovanni Molica~Bisci, and Raffaella Servadei.
\newblock Multiplicity results for fractional {L}aplace problems with critical
  growth.
\newblock {\em Manuscripta Math.}, 155(3-4):369--388, 2018.

\bibitem{GAP}
J.~Garc{\'{\i}}a~Azorero and I.~Peral~Alonso.
\newblock Multiplicity of solutions for elliptic problems with critical
  exponent or with a nonsymmetric term.
\newblock {\em Trans. Amer. Math. Soc.}, 323(2):877--895, 1991.

\bibitem{G-L}
Giambattista Giacomin and Joel~L. Lebowitz.
\newblock Phase segregation dynamics in particle systems with long range
  interactions. {I}. {M}acroscopic limits.
\newblock {\em J. Statist. Phys.}, 87(1-2):37--61, 1997.

\bibitem{G-O}
Guy Gilboa and Stanley Osher.
\newblock Nonlocal operators with applications to image processing.
\newblock {\em Multiscale Model. Simul.}, 7(3):1005--1028, 2008.

\bibitem{Hu}
Nicolas et~al. Humphries.
\newblock Environmental context explains l\'evy and brownian movement patterns
  of marine predators.
\newblock {\em Nature}, 465:1066--1069, 2010.

\bibitem{Leve}
Sergei Levendorski.
\newblock Pricing of the american put under l{\'e}vy processes.
\newblock {\em International Journal of Theoretical and Applied Finance},
  7(03):303--335, 2004.

\bibitem{Li}
Yuanyuan Li.
\newblock The existence of solutions for quasilinear elliptic problems with
  multiple {H}ardy terms.
\newblock {\em Appl. Math. Lett.}, 81:7--13, 2018.

\bibitem{Lions}
P.-L. Lions.
\newblock The concentration-compactness principle in the calculus of
  variations. {T}he limit case. {I}.
\newblock {\em Rev. Mat. Iberoamericana}, 1(1):145--201, 1985.

\bibitem{Ma-Va}
Annalisa Massaccesi and Enrico Valdinoci.
\newblock Is a nonlocal diffusion strategy convenient for biological
  populations in competition?
\newblock {\em J. Math. Biol.}, 74(1-2):113--147, 2017.

\bibitem{MS}
Sunra Mosconi and Marco Squassina.
\newblock Nonlocal problems at nearly critical growth.
\newblock {\em Nonlinear Anal.}, 136:84--101, 2016.

\bibitem{PSY}
Kanishka Perera, Marco Squassina, and Yang Yang.
\newblock Bifurcation and multiplicity results for critical fractional
  {$p$}-{L}aplacian problems.
\newblock {\em Math. Nachr.}, 289(2-3):332--342, 2016.

\bibitem{Re-Rho}
A.~M. Reynolds and C.~J. Rhodes.
\newblock The l\'evy flight paradigm: random search patterns and mechanisms.
\newblock {\em Ecology}, 90(4):877--887, 2009.

\bibitem{Sch}
Wim Schoutens.
\newblock {\em L\'evy Processes in Finance: Pricing Financial Derivatives}.
\newblock Willey Series in Probability and Statistics. Willey, New York, 2003.

\bibitem{S2}
Raffaella Servadei.
\newblock The {Y}amabe equation in a non-local setting.
\newblock {\em Adv. Nonlinear Anal.}, 2(3):235--270, 2013.

\bibitem{S1}
Raffaella Servadei.
\newblock A critical fractional {L}aplace equation in the resonant case.
\newblock {\em Topol. Methods Nonlinear Anal.}, 43(1):251--267, 2014.

\bibitem{SV1}
Raffaella Servadei and Enrico Valdinoci.
\newblock A {B}rezis-{N}irenberg result for non-local critical equations in low
  dimension.
\newblock {\em Commun. Pure Appl. Anal.}, 12(6):2445--2464, 2013.

\bibitem{SV2}
Raffaella Servadei and Enrico Valdinoci.
\newblock Fractional {L}aplacian equations with critical {S}obolev exponent.
\newblock {\em Rev. Mat. Complut.}, 28(3):655--676, 2015.

\bibitem{Silva}
Analia Silva.
\newblock Multiple solutions for the {$p(x)$}-{L}aplace operator with critical
  growth.
\newblock {\em Adv. Nonlinear Stud.}, 11(1):63--75, 2011.

\bibitem{st}
Michael Struwe.
\newblock Three nontrivial solutions of anticoercive boundary value problems
  for the pseudo-{L}aplace operator.
\newblock {\em J. Reine Angew. Math.}, 325:68--74, 1981.

\bibitem{T}
Hans Triebel.
\newblock {\em Theory of function spaces}.
\newblock Modern Birkh\"auser Classics. Birkh\"auser/Springer Basel AG, Basel,
  2010.
\newblock Reprint of 1983 edition [MR0730762], Also published in 1983 by
  Birkh\"auser Verlag [MR0781540].

\end{thebibliography}

\end{document}